\documentclass[12pt]{amsart}
\usepackage[normalem]{ulem}
\usepackage{yfonts, amsfonts}
\usepackage[a4paper]{geometry}
\usepackage[centertags]{amsmath}
\usepackage{amsmath, amsopn, amsthm, amssymb}
\usepackage{graphicx, subfigure}
\usepackage{times}
\usepackage{color}
\usepackage[mathcal]{euscript}
\usepackage{mathtools}
\usepackage{bbold}
\usepackage{hyperref}
\usepackage{mathrsfs}
\usepackage{cite}

\newtheorem{proposition}{Proposition}[section]

\newtheorem{theorem}{Theorem}[section]%

\newtheorem{lemma}[proposition]{Lemma}

\theoremstyle{remark}
\newtheorem{remark}{Remark}

\newcommand{\eps}{\varepsilon}

\newcommand{\RR}{\mathbb {R}}
\newcommand{\NN}{\mathbb {N}}

\DeclareMathOperator{\y}{\mathbf {y}}
\DeclareMathOperator{\f}{\mathbf{f}}
\DeclareMathOperator{\z}{\mathbf{z}}

\DeclareMathOperator{\id}{I}

\begin{document}
\title[On the stability and null-controllability of a linear infinite  system]{On the stability and null-controllability of an infinite system of linear differential equations}
\author[A. Azamov, G. Ibragimov, K. Mamayusupov, M. Ruziboev]{Abdulla Azamov, Gafurjan Ibragimov, Khudoyor Mamayusupov, \\Marks Ruziboev}
\address[Abdulla Azamov]{Section of Dynamical Systems and Their Applications, V.I.Romanovskiy Institute of Mathematics, Uzbek Academy of Sciences, 4, University street, Olmazor, Tashkent, 100174, Uzbekistan}
\email{abdulla.azamov@gmail.com}
\address[Gafurjan Ibragimov]{Department of Mathematics and Institute for Mathematical Research,
University Putra Malaysia, Malaysia}
\email{ibragimov@upm.edu.my}
\address[Khudoyor Mamayusupov]{Moscow Institute of Physics and Technology, Institutsky lane 9, Dolgoprudny, Moscow region, 141700, Russia\newline The National University of Uzbekistan,  4, University street, Tashkent, 100174, , Uzbekistan}
\email{mamayusupov@phystech.edu}
\address[Marks Ruziboev]{Faculty of Mathematics, University of Vienna, Oskar-Morgnstern Platz 1, Austria}
\email{marks.ruziboev@univie.ac.at}
\thanks{MR would like to thank The FWF for  supporting his research through the Lise Meitner Fellowship FWF M-2816. The authors also thank the referees for useful comments which improved the exposition.}
\maketitle

\begin{abstract}
In this work, the null controllability problem for a linear system in $\ell^2$ is considered, where the matrix of a linear operator describing the system is an infinite matrix with $\lambda\in \RR$ on the main diagonal and 1s above it. We show that the system is asymptotically stable if and only if  $\lambda\le-1$, which shows the fine difference between the finite and the infinite-dimensional systems. When $\lambda\le-1$ we also show that the system is null controllable in large. We also show a dependence of the stability on the norm i.e. the same system considered in $\ell^\infty$ is not asymptotically stable if $\lambda=-1$. 
\end{abstract}

\section{Statement of the problem}\label{sec:sofprob}
Control problems in Banach or Hilbert spaces arise naturally in processes described by partial differential equations (see for example \cite{Lions, CurZwart, Fur,  WWX, AS, ErvZua, BMZ, CMZ, CorXinag} and references therein).  Sometimes it is useful to reduce the control problem for partial differential equations to infinite systems of ODEs \cite{Cher90, Cher, AzRuz, AzBakAkh}. Also, it is of independent interest to consider control systems governed by infinite system as models in Banach spaces. For example in  \cite{SatimovTukht2007, TukhtMamatov2008} control problems for infinite systems are considered. 

A considerable amount of work devoted to  differential game problems for infinite systems in Hilbert spaces (see for example \cite{Idham_IGI_Ask2016, Ibragimov-Al-Kuch2014} and references therein). Optimal strategies for players in suitable classes of strategies have been constructed in  \cite{Ibragimov-ScAs2013}. 

Often it is useful to study finite dimensional approximations of the infinite system, such an approach is taken in \cite{AzRuz, AzBakAkh}. The main difficulty is then to prove that the approximate solutions converge to a solution of the initial control problem. In the above works the authors obtain infinite  linear ODEs, where the right hand has  a diagonal form. Hence it is not difficult to show that finite dimensional approximations converge to the solutions of the original system in a suitable sense.
The proofs suggest that similar results maybe proven for linear systems with block diagonal form under certain mild assumptions. 

In fact, as it is shown in \cite{ZMI} for certain  linear systems with quadratic cost there are  approximation schemes that converge, but the approximating controls do not even stabilize the original system and also the costs does not converge. 

In this work we consider a simple infinite linear controllable system in $\ell^2$. The main feature of the system is that it is an infinite Jordan  block, with $\lambda\in \RR$ on the main diagonal. Therefore, any finite dimensional approximation of the system is asymptotically stable whenever $\lambda<0$, but the infinite system is stable  if and only if $\lambda\le-1$ and when $\lambda>-1$ solutions in certain directions grow exponentially fast. This shows fine difference between finite dimensional and infinite systems. Another main feature of this notes is that using Gramian operators, we give explicit form of control functions that stabilize the system. 

In the rest of this section we formulate the problem and state the main results. In section \ref{sec:proof} we prove global asymptotic stability. In section \ref{sec:control} we show global null-controllability and in Section \ref{sec:disc} we discuss the results and further generalizations.
  
Let $\ell^2=\{\y=(y_1, y_2, \dots )\mid y_n\in \RR, \sum_{n\ge1} y_n^2< \infty\}$. We consider $\ell^2$ with it's natural norm: $\|\y\|_2^2=\sum_{n\ge1} y_n^2$, which turns it into a Hilbert space.
\medskip

Given an infinite system of  ODEs:
\begin{equation}\label{eq:syst1}
\dot{y}_{n}=\lambda y_{n}+y_{n+1},  y_n(0)=y_{n,0},
\end{equation}
where  $\lambda\in\RR$ is a fixed number and  $\y_0=\{y_{n,0}\}_{n\in\NN}\in \ell^2$.
We can rewrite the system in an operator form
\begin{equation}\label{eq:1}
\dot{\y}=A\y, \y(0)=\y_0,
\end{equation}
where $\y_0=\{y_{n,0}\}_{n\in\NN}$ and $A:\ell^2\to \ell^2$ is a linear operator defined by
\[A\y=\{\lambda y_n+y_{n+1}\}_{n\in \NN}.\]

This is an example of an ODE in a Banach space, which is a well studied topic (see for example \cite{Dei, CurZwart}), here we study the stability and control problems. In particular, we construct controls function explicitly.

Observe that $A$ is a bounded linear operator, in fact we have
\[\begin{aligned}
\|A\y\|_2^2=\sum_{n\ge 1} (\lambda y_n+y_{n+1})^2
 \le (1+|\lambda|)^2\|\y\|_2^2.
\end{aligned}\]
Hence, $\|A\|=\sup_{{\|\y\|}_{2}=1}\|A\y\|_2\le 1+|\lambda|$.

Now, it is standard to define $e^{tA}$ as
\[
e^{tA}:=\sum_{n\ge 0}\frac{t^nA^n}{n!},
\]
which is bounded on  $\ell^2$ for every $t\in \RR$. Further, $e^{tA}$ admits all the properties of analogues operator for matrices. In particular, $e^{tA}$ defines a group of operators.
The solution of \eqref{eq:1} can be written in the form
\[\y(t)=e^{tA}\y_0.\]
We also consider the Cauchy problem for non-homogeneous equation
\begin{equation}\label{eq:controlsyst}
\dot{\y}=A\y+\f, \quad\y(0)=\y_0,
\end{equation}
for $\textbf f:\RR\to \ell^2$, $\f \in L^2([0, T], \ell^2)$, i.e. $\|\f\|_{L^2}^2=\int_{0}^T\|\f(t)\|^2_2 dt<+\infty$\footnote{We note that this norm coincides with $\|\f\|_{L^2}^2=\sum_{n\ge 1}\int_{0}^T |f_n(t)|^2dt$ thanks to Beppo-Levi's theorem.}. 

A function $\y:[0, T]\to \ell^2$ defined as 
\[
\y(t)=e^{tA}\y_0 + e^{tA}\int_{0}^te^{-sA}\f(s)ds
\]
is called a mild solution of \eqref{eq:controlsyst} if $\y\in C([0, T], \ell^2)$.  Here the integration is understood componentwise.
For completeness we start with the following.
\begin{proposition}\label{prop:exst}
For every $\f \in L^2([0, T], \ell^2)$ and $\y_0\in \ell^2$ we have $\y\in C([0, T], \ell^2)$.
\end{proposition}

The next result is about stability. In this simple setting we can characterize the system completely. We have the following.
\begin{proposition}\label{prop:stability}
Let $\y(t)$ be the solution of  \eqref{eq:syst1} with an initial condition $\y_0\in\ell^2$.
 System \eqref{eq:syst1} is asymptotically stable if and only if $\lambda\le-1$. Moreover for every  $\y_0\in\ell^2$ and for every $t\in\RR$ holds $\|e^{tA}\y_0\|_2\le e^{(1+\lambda)t}\|\y_0\|_2$ .
\end{proposition}
%
Let $\rho>0$ be fixed. A control function $\textbf f:\RR\to \ell^2$ is called admissible if 
\[\|\f\|_{L^2}^2=\int_{0}^T\|\f(t)\|^2_2 dt\le \rho^2.\]
We say that the system  \eqref{eq:controlsyst} is \textit{null-controllable} from $\y_0\in\ell^2$  an admissible control $ \textbf f:\RR\to \ell^2$  and $T=T(f)\in\RR$ such that the solution of  \eqref{eq:controlsyst}  satisfies  $\y(T)=0$. 

We say that the system  \eqref{eq:controlsyst} is \textit{locally null-controllable} if there exists $\delta=\delta(\rho)>0$ such that \eqref{eq:controlsyst} is null-controllable from any $\y_0\in\ell^2$ with $\|\y_0\|\le \delta$.

We say that the system  \eqref{eq:controlsyst} is \textit{globally null-controllable} if it is null-controllable from any $\y_0\in\ell^2$.

The main result of this notes is the following 
\begin{theorem}\label{main}
\begin{itemize}
\item[(i)]The system \eqref{eq:controlsyst} is locally null-controllable for every $\lambda\in \RR$.  
\item[(ii)] If  $\lambda\le-1$,  then system \eqref{eq:controlsyst} is globally null-controllable. 
\item[(iii)] If $\lambda<-1$ the systems can be transferred from an initial point $\y_0\in\ell^2$ into the origin for time $\tau\ge {\|\y_0\|^4_2}/{\kappa\rho^4}$, where $\kappa$ is a constant independent of $y_0$.
\end{itemize} 
\end{theorem}

Notice that we didn't aim to state the results in the most general form. Also, in  Proposition \ref{prop:stability} for $-1<\lambda<0$ we construct solutions going to $\infty$ as $t\to+\infty$, i.e. $0$ is not Lyapunov stable. Thus jump when passing from $\lambda=-1$ somewhat unusual. But apparently it is due to the structure of $\ell^2$ and very special structure of $A=\lambda \id+E$, i.e.  the shift operator $E:\ell^2\to \ell^2$ is weakly contracting ($E^n\y\to 0$ as $n\to \infty$ for all $\y\in\ell^2$) in this case. The proofs show that analogues results are true  for all  $\ell^p$ spaces with  $1\le p<+\infty$. However, in $\ell^\infty$ the trivial solution $0$ is Lyapunov stable, but it is not asymptotically stable when $\lambda=-1$, see subsection \eqref{sec:stab}.

\section{Asymptotic stability}\label{sec:proof}

We start this section with the proof of Proposition \ref{prop:exst}. 
\begin{proof}[Proof of Proposition \ref{prop:exst}]
For any $T\in\RR$, $t_0,t\in[0, T]$ and $y_0\in\ell^2$  we have
\begin{equation}\label{eq:et-t0}
\|e^{tA}-e^{t_0A}\|_2\le \|e^{t_0A}\|_2\cdot \|e^{(t-t_0)A}-\id\|_2\le |t-t_0|\cdot\|A\|_2\cdot e^{2T\|A\|_2},
\end{equation}
where in the last inequality we have used the definition of $e^{tA}$ and $\|e^{tA}\|_2\le e^{|t|\cdot \|A\|_2}$. For any $\f \in L^2([0, T], \ell^2)$ and $0\le t_0\le t\le T$ we have
\begin{equation}\label{eq:intt0-t}
\begin{aligned}
\left\|\int_{t_0}^te^{-sA}\f(s)ds\right\|_2 &\le\int_{t_0}^t\|e^{-sA}\|_2\cdot \|\f(s)\|_2ds\\
 &\le\left(\int_{t_0}^t e^{2s \cdot \|A\|_2}\right)^{1/2}\left(\int_{t_0}^t\|\f(s)\|_2^2ds\right)^{1/2}\\
 &\le |t-t_0|^{1/2}e^{t\cdot \|A\|_2}\|\f\|_{L^2},
 \end{aligned}
\end{equation}
where in the second inequality the Cauchy-Schwarz inequality is used.
We have 
\begin{equation}\label{difference}
\begin{aligned}
\|\y(t)-\y(t_0)\|_2\le &\|e^{tA}-e^{t_0A}\|_2\cdot\|\y_0\|_2 \\ &+\|e^{tA}\int_{0}^te^{-sA}\f(s)ds-e^{t_0A}\int_{0}^{t_0}e^{-sA}\f(s)ds\|_2.
\end{aligned}
\end{equation}
The first term on the right hand side of the inequality tends to $0$ when $t\to t_0$ by \eqref{eq:et-t0}. We will show that the second term also tends to $0$ as $t$ approaches $t_0$. Without loss of generality, assume $t>t_0$ then the second term of \eqref{difference} is bounded by
\[
\|e^{tA}\|_2\cdot \|\int_{t_0}^te^{-sA}\f(s)ds\|_2+ \|e^{tA}-e^{t_0A}\|_2\cdot \|\int_{0}^{t_0}e^{-sA}\f(s)ds\|_2,
\]
which tends to $0$ by   \eqref{eq:intt0-t}. The proof is finished.
\end{proof}

\subsection{Stability of the system} \label{sec:stab}
Here we give necessary and sufficient condition for the asymptotic stability of the system  in $\ell^2$. Recall that the system  \eqref{eq:syst1} is called globally asymptotically stable if $\lim_{t\to +\infty} \y(t)=0$  for the solution $\y(t)$ of  \eqref{eq:syst1} with any initial condition $\y_0\in\ell^2$.

\begin{proof}[Proof of Proposition \ref{prop:stability}] We  write $A=\lambda \id+ E$, where
$\id$ is the identity map, and $E:\ell^2\to \ell^2$ is the shift map, i.e. $[Ef]_i=f_{i+1}$. Then we have $e^{tA}=e^{\lambda t}e^{tE}$. Now, we obtain item (i) directly:
\[
\|\y(t)\|_2\le \|e^{tA}\|\cdot \|\y_0\|_2=e^{\lambda t}\|e^{tE}\|\cdot \|\y_0\|_2\le e^{(\lambda+1) t}\cdot \|\y_0\|_2.
\]
If $\lambda<-1$ the latter inequality implies $\lim_{t\to +\infty}\y(t)=0$.   

If $\lambda=-1$ then the above argument doesn't imply the desired conclusion. Thus we proceed as follows.  Observe that 

\begin{equation}\label{eq:etE}
    e^{tE}= \begin{pmatrix}
          1 & t & \frac{t^2}{2!} &\dots &\frac{t^k}{k!} &\dots \\
          0 & 1 &     t          &\dots &\frac{t^{k-1}}{(k-1)!} &\dots \\
          0 & 0 &     1          &\dots & \frac{t^{k-2}}{(k-2)!} &\dots \\
          \vdots& \vdots& \vdots  &\ddots&  \vdots &\ddots
    \end{pmatrix}.
\end{equation}

Then for any $\z\in\ell^2$ with $\|\z\|_2=1$ and for the solution $\y(\cdot)$  started from $\y_0=(y_{10}, y_{20}, \dots)\in\ell^2$ we obtain 
\begin{equation}\label{eq:ytz}
\langle\y(t), \z\rangle = e^{-t} \langle\sum_{j\ge 0}\frac{t^j}{j!}E^j\y_0, \z\rangle= e^{-t}\sum_{j\ge 0} \frac{t^j}{j!} \langle E^j\y_0, \z\rangle.
\end{equation}
From the definition of $E$ we have  $\|E^j\y_0\|_2\le  \|\y_0\|_2$ for all $j\in\NN_0$  and  $\|E^j\y_0\|_2\to 0$ as $j\to \infty$. Thus for any $\eps>0$ there exists $N=N(\y_0)\in\NN_0$ such that $\|E^j\y_0\|_2\le \eps/2$ for all $j\ge N$. Fixing such an $N$ and using $|\langle E^j\y_0, \z\rangle| \le \|E^j\y_0\|_2\le \|\y_0\|_2$ for $j\in \NN$ from \eqref{eq:ytz} we obtain
\begin{equation}\label{eq:|ytz|}
\begin{aligned}
|\langle\y(t), \z\rangle|&\le e^{-t}\sum_{j=0}^N \frac{t^j}{j!}\|E^j\y_0\|_2 + e^{-t}\sum_{j=N+1}^\infty \frac{t^j}{j!} \|E^j\y_0\|_2\\
&\le \|\y_0\|_2e^{-t} \sum_{j=0}^N \frac{t^j}{j!}+ \frac{\eps}2e^{-t}\sum_{j=0}^\infty \frac{t^j}{j!}\le \|\y_0\|_2C_Ne^{-t/2}+\frac{\eps}2.
\end{aligned}
\end{equation}
Notice that the choice of $N$ and hence $C_N$ is independent of $t$. Therefore,  \eqref{eq:|ytz|} implies that there exists $ t(\eps)>0$ such that $\|\y_0\|_2C_Ne^{-t/2}\le \eps/2$ for all $t\ge t(\eps)$. Finally,   taking $\z=\y(t)/\|\y(t)\|_2$ in \eqref{eq:|ytz|} results to 
\[
\|\y(t)\|_2\le \eps \text{ for all } t\ge t(\eps).
\]
This finishes the proof of item (ii). 

Now we show item (iii). Suppose that $\lambda>-1$. 
Since for  $\lambda>0$ the system \eqref{eq:syst1}  is  not stable, it suffices to consider the case $-1<\lambda\le 0$. 
Let  $\theta\in(0, 1)$ and $\Theta=(1, \theta, \theta^2, \theta^3, \dots )$. Obviously, $\Theta\in\ell^2$ and $e^{tE}\Theta=e^{t\theta}\Theta$. Since 
$-1<\lambda\le0$ if we let $\theta= -\lambda+\frac{|1+\lambda|}{2}\in (0, 1)$ then as $t\to +\infty$ one gets
\begin{equation}\label{eq:etatheta}
\|e^{tA}\Theta\|_2=e^{t\lambda}\|e^{tE}\Theta\|_2=e^{t|1+\lambda|/2}\|\Theta\|_2\to +\infty.
\end{equation}
This implies that if  $\lambda>-1$ then \eqref{eq:syst1} is not stable.
This completes the proof. 
\end{proof}

\begin{remark}
Consider the system $\dot\y=A\y$, $\y(0)=\y_0\in\ell^\infty$, where 
\[
\ell^\infty=\{\y=(y_1, y_2, \dots)\mid \sup_{n\in\NN}|y_n|<\infty\}.
\]
Then $\textbf{e}=(1,1,1\dots)\in\ell^\infty$ is an eigenvector of $e^{tE}$ corresponding to the eigenvalue $e^{t}$. Thus $0$ is Lyapunov stable but it is not asymptotically stable. 
\end{remark}

\section{Null-controllability}\label{sec:control}
Here we show that  system \eqref{eq:controlsyst} is null controllable.

We start with  a standard lemma from operator theory, which will be useful below. 
\begin{lemma}\label{lem:inv}
Let $L:\mathcal H\to \mathcal H$ be a self adjoint operator defined on a Hilbert space $(\mathcal H, \|\cdot\|)$. Assume that
there exists $\kappa>0$ such that $\|Lx\|\ge \kappa\|x\|$ for all $x\in L$. Then  $L$ is invertible and $\|L^{-1}\|\le \kappa^{-1}$.
\end{lemma}
To prove controllability we use Gramian operators and prove an observability inequality. For $\tau\in \RR$ define 
\[
W(\tau)=\int_0^\tau e^{-sA}\cdot e^{-sA^\ast}ds,
\]
where  and $A^\ast$ is the adjoint of $A$ in $\ell^2$. The following lemma is the main technical tool 
\begin{lemma}
For every $\tau\in \RR$ the operator $W(\tau)$ is bounded, self adjoint, positive definite and invertible. Moreover, there exists $\kappa >0$ such that $\|W(\tau)\y\|_2\ge \kappa\|\y\|_2$ for any $\y\in\ell^2$. 
\end{lemma} 

\begin{proof}
One can easily verify that $E^\ast\f=(0, \f)=(0, f_1, f_2, \dots)$.  Then $e^{tA}e^{tA^\ast}=e^{2t\lambda}e^{tE}\cdot e^{tE^\ast}$. Further,
\begin{equation}\label{eq:|Wxy|}
    \begin{aligned}
& |\langle W(\tau)\y, \z\rangle |\le\| \int_0^\tau e^{-2t\lambda}e^{-tE}\cdot e^{-tE^\ast}\y dt\|_2\cdot \|\z\|_2 \le\\ 
& \le  \int_0^\tau e^{2t(1-\lambda)} dt \cdot\|\y\|_2\cdot \| \z\|_2\le M(\tau)\cdot \|\y\|_2\cdot \| \z\|_2,
 \end{aligned}
 \end{equation}
 where the constant $M(\tau)$ depends only on $\tau$.

Let $e_{ij}(t)$, $i,j\in \NN$, denote an element of $e^{tE}\cdot e^{tE^\ast}$. For $\y,\z\in\ell^2$ we have
\begin{equation}\label{eq:Wxy}
   \langle W(\tau)\y, \z\rangle = \sum_{i=1}^\infty\sum_{j=1}^\infty\int_0^\tau e^{-2t\lambda}e_{ij}(-t)y_jz_idt.
\end{equation}
By \eqref{eq:|Wxy|} the right hand side of \eqref{eq:Wxy} is absolutely convergent. Thus,
\begin{equation*}
   \langle W(\tau)\y, \z\rangle = \sum_{j=1}^\infty\sum_{i=1}^\infty\int_0^\tau e^{-2t\lambda}e_{ij}(-t)y_jz_idt =  \langle \y, W(\tau)\z\rangle.
\end{equation*}
This implies that $W(\tau)$ is self adjoint for every $\tau\in\mathbb R$.

Notice that $e^{tE^\ast}$ is just the transpose of $e^{tE}$. Therefore, by \eqref{eq:etE} for $i,j\in \NN$ we have
\begin{equation*}
e_{ij}(t)=\sum_{m=|i-j|}^\infty\frac{t^m}{m!}\cdot\frac{t^{m-|i-j|}}{(m-|i-j|)!},
\end{equation*}
which implies that both of the series
\begin{equation*}
\sum_{i=1}^\infty e_{ij}(-t)y_jz_i \quad\text{and}\quad  \sum_{j=1}^\infty\sum_{i=1}^\infty e_{ij}(-t)y_jz_i
\end{equation*}
converge uniformly in $[0, \tau]$, hence
\begin{equation}\begin{aligned}\label{eq:Wyz}
  & \langle W(\tau)\y, \z\rangle = \sum_{j=1}^\infty\sum_{i=1}^\infty\int_0^\tau e^{-2t\lambda}e_{ij}(-t)y_jz_idt \\
   &=\int_0^\tau\sum_{j=1}^\infty\sum_{i=1}^\infty e^{-2t\lambda}e_{ij}(-t)y_jz_idt =\int_0^\tau e^{-2t\lambda}\langle  e^{-tE^\ast}\y,e^{-tE^\ast}\z\rangle dt.
\end{aligned}
\end{equation}

The above equation immediately implies that $\langle W(\tau)\y, \y \rangle>0$ for every $\y \neq 0$ i.e. $W(\tau)$ is positive definite. In \eqref{eq:Wyz} we have showed that we can take integration out of the scalar product $\langle W(\tau)\y, \z\rangle$. We will use this property several times below.

For every $\varepsilon\in [0, \tau]$
\[
\langle W(\tau)\y, \y \rangle \ge\int_0^{\varepsilon} e^{-2t\lambda} \langle e^{-tE^\ast}\y,e^{-tE^\ast}\y\rangle dt.
\]
Now we look at the operator $e^{-tE}\cdot e^{-tE^\ast}$. Note that $EE^\ast=\id$, we have
\begin{equation}\label{eq:expet}
e^{-tE}\cdot e^{-tE^\ast}=\sum_{n=0}^\infty\frac{t^{2n}}{(n!)^2}\id+ \sum_{n=0}^\infty \sum_{m=n+1}^\infty \frac{(-t)^{n+m}}{n!m!}(E^{m-n}+{(E^\ast)}^{m-n}).
\end{equation}
It follows that for sufficiently small $\varepsilon>0$ and $t\in(0, \eps)$ we have
\[
e^{-tE}\cdot e^{-tE^\ast}=\id-t(E+E^\ast)+o(t),
\]
where $o(t)$ is a linear operator whose $\ell^2$ norm is $o(t)$ in the usual sense. Finally,
\begin{equation}\begin{aligned}\label{eq:weakes}
\int_0^\varepsilon e^{-2t\lambda} \langle e^{-tE}\cdot e^{-tE^\ast}\y, \y\rangle dt=\int_0^\varepsilon e^{-2t\lambda} \langle (\id -t(E+E^\ast)+o(t))y, y\rangle dt\\
>(1-3\varepsilon)\|y\|_2^2\int_0^\varepsilon e^{-2t\lambda}dt=\frac{1-3\varepsilon}{-2\lambda}(e^{-2\lambda\varepsilon}-1)\|y\|_2^2,
\end{aligned}
\end{equation}
where we used $\langle E\y, \y \rangle\le \|y\|_2^2$.  This proves
\begin{equation}\label{eq:kappa}
    \|W(\tau)\y\|_2\ge \kappa\|\y\|_2, \text{ with } \kappa^2=\frac{1-3\varepsilon}{-2\lambda}(e^{-2\lambda\varepsilon}-1)>0.
\end{equation}
Thus Lemma \ref{lem:inv} is applicable and implies that  $W(\tau)$ is invertible for every $\tau >0$ and $W^{-1}(\tau):\ell^2\to \ell^2$ is a bounded linear operator with the norm $\|W(\tau)^{-1}\|\le \kappa^{-1}$, where  $\kappa$ is independent of $\tau$.
\end{proof}
Now we are ready to prove Theorem \ref{main}. 

\begin{proof}[Proof of Theorem \ref{main}] 
Below we assume that $\rho>0$ and the set of admissible control is defined as in Section \ref{sec:sofprob}. 
Recall that $\y(t)=e^{tA}\y_0 + e^{tA}\int_{0}^te^{-sA}\f(s)ds$  is the unique solution of  system \eqref{eq:controlsyst} with an initial state $\y(0)=\y_0$.

We look for a solution of the control problem in the form
\begin{equation}\label{eq:contol}
    \f_0(t)=-e^{-tA^\ast}\cdot W^{-1}(\tau)\y_0 \quad\text{for every}\quad \y_0\in \ell^2, \tau\in\RR^+.
\end{equation}

We  show that $\int_0^\tau e^{-sA}\f_0(s)ds=-\y_0$ for every fixed $\tau\in \RR^+$. Indeed, by \eqref{eq:Wyz} we have
\begin{equation}\label{eq:int=y_0}
-\int_0^\tau e^{-tA}\f_0dt=
\int_0^\tau e^{-tA}e^{-tA^\ast}dt\cdot W^{-1}(\tau)\y_0= \y_0.
\end{equation}
It remains to show that $\f_0$ is admissible,  i.e.  there exists $\tau>0$ such that  $\|\f_0\|_{L^2}\le \rho$.

By definition of $W(\tau)$ and \eqref{eq:Wyz} we have
\begin{equation}\label{eq:norm}
\begin{aligned}
   \int_0^\tau\|\f_0(t)\|^2_2dt&=\int_0^\tau \|e^{-tA^\ast}W^{-1}(\tau)\y_0\|_2dt\\
   &= \int_0^\tau \left\langle e^{-tA}\cdot e^{-tA^\ast}W^{-1}(\tau)\y_0, W^{-1}(\tau)\y_0\right\rangle dt\\
    &=\langle \y_0, W^{-1}(\tau)\y_0\rangle\le \|\y_0\|_2\cdot\|W^{-1}(\tau)\y_0\|_2.
\end{aligned}
\end{equation}

To prove item (i) we look for the set of $\y_0\in\ell^2$ with $\|\y_0\|_2^2\le \kappa\rho^2$. Then by \eqref{eq:int=y_0} we have that $\y(\tau)=0$ for the solution started from $\y_0$. Also, by \eqref{eq:norm} and the choice of $\y_0$ the function $\f_0$ defined by \eqref{eq:contol} is admissible. 
 
To prove item (ii) we consider cases $\lambda<-1$ and $\lambda=-1$ separately. 

\textbf{Global null-controllability for $\lambda<-1$.} We will prove that $\|W^{-1}(\tau)\y_0\|_2\to 0$ as $\tau\to +\infty$.  To this end  we refine the inequality in \eqref{eq:weakes} as follows.

Since $e^{tA^\ast}$ is invertible,
\[
\|\y\|_2=\|e^{tA^\ast}e^{-tA^\ast}\y\|_2\le \|e^{tA^\ast}\|\cdot \|e^{-tA^\ast}\y\|.
\]
Thus, by Proposition \ref{prop:stability} we have
\[
 \|e^{-tA^\ast}\y\|_2\ge e^{-t(1+\lambda)}\|\y\|_2.
\]
Consequently, for any $\y\in \ell^2$ holds
\[
\langle W(\tau)\y, \y\rangle= \int_0^\tau \|e^{-tA^\ast}\y\|^2_2dt\ge \int_0^\tau e^{-2t(1+\lambda)}\|\y\|^2_2dt
= \|\y\|^2_2\cdot\frac{e^{-2(1+\lambda)\tau}-1}{-2(1+\lambda)}.
\]
Recalling  $\|W^{-1}(\tau)\|\le \kappa^{-1}$ and   letting $\z(\tau)=W(\tau)^{-1}\y_0\in\ell^2$ by the above inequality  we have
\[
\kappa^{-1}\|\y_0\|^2_2\ge \langle \y_0, W^{-1}(\tau)\y_0\rangle=\langle W(\tau)\z(\tau), \z(\tau)\rangle\ge \|\z(\tau)\|^2_2\cdot\frac{e^{-2(1+\lambda)\tau}-1}{-2(1+\lambda)}.
\]
Hence,
\begin{equation}\label{eq:||ztau||}
\|\z(\tau)\|_2\le \left(\frac{-2(1+\lambda)}{\kappa(e^{-2(1+\lambda)\tau}-1)}\right)^{1/2}\|\y_0\|_2.
\end{equation}
Since $\lambda<-1$ the right hand side of the above inequality converges to $0$ exponentially fast as $\tau\to +\infty$ and so does $\|\z(\tau)\|_2$. Thus, by \eqref{eq:norm} there exists $\tau_0$ such that 
\[
 \int_0^\tau\|\f_0(t)\|^2_2dt\le \rho^2 \quad\text{for all}\quad\tau>\tau_0.
\]
This finishes the proof of global controllability for $\lambda<-1$. 

\textbf{Global null-controllability for $\lambda=-1$.}
This case needs a slightly different argument. Recall that  that in this case the system is locally null controllable i.e.  the control function defined in \eqref{eq:contol} remains admissible in the neighbourhood of the origin: if $\|\y_0\|_2\le \rho\sqrt\kappa$, where $\kappa$ is the constant defined in \eqref{eq:kappa}, we set \begin{equation}\label{eq:contol1}
    \f_1(t)=-e^{-tA^\ast}\cdot W^{-1}(1)\y_0 \quad\text{for every}\quad \y_0\in \ell^2.
\end{equation}
Then  by \eqref{eq:norm} we get
\begin{equation*}
\int_0^\tau\|\f_1(t)\|^2_2dt\le \|\y_0\|_2^2\cdot \|W^{-1}(1)\| \le \rho^2,
\end{equation*}
and 
\[
\y(1)=e^{A}\y_0 + e^{A}\int_{0}^1e^{-sA}\f(s)ds=0.
\]
Further, by stability of the system \eqref{eq:syst1} for any $\y_0\in \ell^2$ there exists $\tau_0=\tau(\kappa, \rho, \y_0)$ such that $\|e^{tA}\y_0\|_2\le \rho\sqrt\kappa$ for any $t\ge \tau_0$. Therefore, we set 
\begin{equation}\label{eq:l=1}
\f_0(t)=\begin{cases}
0, \text{ if } t\le \tau_0,\\
\f_1(t), \tau_0\le t\le \tau_0+1. 
\end{cases}
\end{equation}
One can easily check that $\f_0$ is admissible and $\y(\tau_0+1)=0$ for the corresponding solution of \eqref{eq:controlsyst},  which finishes the proof.  This finishes the proof of item (ii)

Observe that to prove the item (iii) it is sufficient to obtain estimates on $\tau$ satisfying 
\[\|\y_0\|_2\cdot \|\z(\tau)\|_2\le \rho^2,\]
where $z(\tau)$ is given by \eqref{eq:||ztau||}, which is equivalent to 
\[
 \left(\frac{-2(1+\lambda)}{\kappa(e^{-2(1+\lambda)\tau}-1)}\right)^{1/2}\|\y_0\|_2^2\le \rho^2, 
\]
which is satisfied if 
\[
\tau\ge  \frac{\|\y_0\|^4_2}{\kappa\rho^4}\ge \frac{1}{2|\lambda+1|} \log\left(1+\frac{2|\lambda+1|}{\kappa}\frac{\|\y_0\|^4_2}{\rho^4}\right).
\]
This completes the proof of the Theorem.
\end{proof}

\section{Discussion of the results and further questions}\label{sec:disc}
In this paper we addressed an infinite system of linear ODEs with a special operator $A=\lambda \id+ E$ on the right hand side. We obtained stability and controllability of the system when $\lambda\le -1$. Initially, the main motivation for this choice was to construct an example whose finite dimensional projections having qualitatively different behavior than the system itself. In the proofs we used Gramian operators, which raised a natural question whether or not the constructed control  functions are optimal, since in the finite dimensional setting this method is known to produce optimal control. In the setting of the current paper,  when $\lambda<-1$ we expect to obtain optimal control. But we were unable to find an analogue of a general result in the spirit of (for example,  \cite[Propostion 2.]{Ibragimov-ScAs2013}), in the infinite dimensional  setting; when $\lambda=-1$  we don't control the system until it gets closer to the origin. Therefore,  we don't expect to obtain optimal control.  Notice that, in the proofs we used the special form of $A$. 
It would be interesting to obtain similar results for more general system 
\begin{equation}\label{eq:gen}
\dot{\y}=A\y+B\f, \quad\y(0)=\y_0,
\end{equation}
where  $A:\ell^2\to \ell^2$  is a bounded operator, and $B:\mathcal L\to \mathcal L$ is an operator from (possible finite dimensional) subspace $\mathcal L$ of $\ell^2$. The proofs suggest that if $B$ is identity and the spectrum of $A$ lies on the left hand side of the imaginary axes, then \eqref{eq:gen} is  globally asymptotically stable. Invertibility of the Gramians seems also to work since it is a perturbative argument.  But for the global null controllability, one needs different estimates to the inverses of the Gramians, or another approach is needed. However, for general 
$B$ the situation is unclear, it would be nice to obtain a similar conditions 
to the classical Kalman (See for example, \cite[Theorem 1.16]{Coron})  or 
an analog of Fattorini-Hautus but in both situations, it isn't clear what should be the exact conditions. Since for Kalman condition injectivity of an operator isn't sufficient for invertibility, and for Fattorini-Hautus usually one assumes countable spectrum with certain properties  (see for example \cite{BadTaka} and references therein).







\begin{thebibliography} {99}
\bibitem{AlbAli} S. Albeverio, S.A. Alimov, On a time-optimal control problem associated with the heat exchange process. Appl. Math. Optim., 57, 58-68, (2008).

\bibitem{AS} A.A. Agrachev, A. V. Sarychev, Controllability of 2D Euler and Navier-Stokes equations by degenerate forcing. Commun. Math. Phys. 265, No. 3, 673-697 (2006).

\bibitem{AzRuz} A.A. Azamov, M.B. Ruziboev, The time-optimal problem for evolutionary partial differential equations. J. Appl. Math. Mech. 77, 220-224, (2013).

\bibitem{AzBakAkh} A.A. Azamov, J.A. Bakhramov, O.S. Akhmedov,
On the Chernous'ko time-optimal problem for the equation of heat conductivity in a rod, Ural Math. J. 5, No. 1, 13-23 (2019).

\bibitem{BadTaka} M. Badra, T. Takahasi, On the Fattorini criterion for approximate controllability and stabilizability of parabolic systems,
ESAIM Control Optim. Calc. Var., vol 20, 3, 924--956, (2014).

\bibitem{BMZ} U.  Biccari, M. Warma, E.  Zuazua,
Controllability of the one-dimensional fractional heat equation under positivity constraints, Commun. Pure Appl. Anal. 19, No. 4, 1949-1978, (2020).

\bibitem{CMZ} E. Cerpa,  C. Montoya  B. Zhang,
{Local exact controllability to the trajectories of the Korteweg–de Vries–Burgers equation on a bounded domain with mixed boundary conditions},  Journal of Differential Equations,   vol. {268},   pp. {4945-4972}, (2020).

\bibitem{Cher90} F.  Chernous'ko, Decomposition and suboptimal control in dynamical systems, Journal of Applied Mathematics and Mechanics,   vol. {54}, pp. 727-734, (1990).

\bibitem{Cher}  F. Chernous'ko, Bounded controls in distributed-parameter systems, Journal of Applied Mathematics and Mechanics, vol.  {56}, {5},
pp.  {707-723}, (1992).

\bibitem{CorXinag} J. Coron, Shengquan Xiang, Small-time global stabilization of the viscous Burgers equation with three scalar controls,
Journal de Math{\'e}matiques Pures et Appliqu{\'e}es, (2021).

\bibitem{Coron} J. Coron,  Control and nonlinearity.
Mathematical Surveys and Monographs, 136. American Mathematical Society, Providence, RI, 2007. xiv+426 pp.

\bibitem{CurZwart} R. F. Curtain, H. Zwart, An Introduction to Infinite-Dimensional Linear Systems Theory, New York, Springer-Verlag, xviii, 698 p. (1995).

\bibitem{Dei} K. Deimling, Ordinary differential equations in banach spaces, Lecture Notes in Mathematics, vol. 596, Springer-Verlag, New
York, NY, (1977).

\bibitem{ErvZua} S. Ervedoza, E. Zuazua, 
Numerical approximation of exact controls for waves. SpringerBriefs in Mathematics. New York,  Springer. xvii, 122 p. (2013).

\bibitem{Fur} A.V. Fursikov, Optimal Control of Distributed Systems. Theory and Applications, American Mathematical Society, Providence, RI, 2000.

\bibitem{Idham_IGI_Ask2016}
I. A. Alias, G. Ibragimov, A. Rakhmanov.   Evasion Differential Game of Infinitely Many Evaders from Infinitely Many Pursuers in Hilbert Space.  Dynamic Games and Applications. 10.1007/s13235-016-0196-0, 6(2): 1--13, (2016).

\bibitem{Ibragimov-Al-Kuch2014}
 G.I. Ibragimov, F. Allahabi,  A.Sh. Kuchkarov, A pursuit problem in an infinite system of second-order differential equations. Ukrainian Mathematical Journal. 65(8): 1203--1216, (2014).

\bibitem{Ibragimov-ScAs2013}
 G.I. Ibragimov,  Optimal pursuit time for a differential game in the Hilbert space $l_2$. Science Asia. 39S: 25--30, (2013).

\bibitem{Lions} J.L. Lions, Controle Optimal de Syst\'emes Gouvern\'ees par des Equations aux D\'eriv\'ees Partielles, Dunod: Paris, France, (1968).



\bibitem{SatimovTukht2007}
 N.Yu. Satimov, M. Tukhtasinov, On Game Problems for Second-Order Evolution Equations. Russian Mathematics. 51(1): 49--57,  (2007).

%
%

\bibitem{TukhtMamatov2008}
M. Tukhtasinov, M.Sh.  Mamatov,  On Pursuit Problems in Controlled Distributed Parameters Systems. Mathematical Notes. 84(2): 256--262, (2008).

\bibitem{WWX} Wang, Gengsheng, Wang, Lijuan, Xu, Yashan, Zhang, Yubiao,  Time optimal control of evolution equations. Progress in Nonlinear Differential Equations and Their Applications: Subseries in Control 92. Cham: Birkhäuser. xvi, 334 p. (2018).


\bibitem{ZMI} H. Zwart,  K. A. Morris, O.V. Iftime, Optimal linear-quadratic control of asymptotically stabilizable systems using approximations,
Systems \& Control Letters, Vol. 146,  p. 8, (2020),
\end{thebibliography}
\end{document}